\documentclass{article}
\usepackage{amsmath}
\usepackage{amsthm}
\usepackage{mathrsfs}
\usepackage{amscd}
\usepackage{amssymb}
\usepackage[all]{xy}

\newcommand{\sHom}{\underline{\mathrm{Hom}}}
\newcommand{\iHom}{\mathscr{H}om}
\newcommand{\colim}[1]{\mathop{\underset{#1}{\mathrm{colim}}}\nolimits}
\newcommand{\psh}{\mathrm{PSh}}
\newcommand{\sh}{\mathrm{Sh}}
\newcommand{\pcsh}{\mathrm{PCSh}}
\newcommand{\csh}{\mathrm{CSh}}
\newcommand{\igrpd}{\mathbb{S}}
\newcommand{\set}{\mathbf{Set}}
\newcommand{\C}{\mathcal{C}}
\newcommand{\D}{\mathcal{D}}
\newcommand{\V}{\mathcal{V}}
\newcommand{\X}{\mathcal{X}}
\newcommand{\F}{\mathscr{F}}
\newcommand{\G}{\mathscr{G}}
\newcommand{\R}{\Lambda}
\newcommand{\p}[2]{\left\langle{#1},{#2}\right\rangle}
\renewcommand{\d}{\dagger}
\newcommand{\e}{\varepsilon}
\renewcommand{\i}{\infty}
\renewcommand{\1}{\mbox{1}\hspace{-0.25em}\mbox{l}}
\newcommand{\yo}{\mathrm{Y}}
\newcommand{\bu}{\bullet}
\renewcommand{\P}{\mathcal{P}}
\newcommand{\U}{\mathcal{U}}

\newtheorem{thm}{Theorem}[section]
\newtheorem{thmi}{Theorem}
\newtheorem{lem}[thm]{Lemma}
\newtheorem{prop}[thm]{Proposition}
\theoremstyle{definition}
\newtheorem{defi}[thm]{Definition}
\newtheorem{examp}[thm]{Example}
\newtheorem{rem}[thm]{Remark}
\newtheorem*{nota}{Notation}
\newtheorem*{acknow}{Acknowledgments}

\makeatletter

\@addtoreset{equation}{section}
\makeatother

\begin{document}

\title{$\infty$-cosheafification }
\author{Yuri Shimizu\footnote{Department of Mathematics, Tokyo Institute of Technology, Tokyo, Japan \newline E-mail: qingshuiyouli27@gmail.com}}
\maketitle

\begin{abstract}
Cosheaves are a dual notion of sheaves. In this paper, we prove existence of a dual of sheafifications, called \textit{cosheafifications}, in the $\infty$-category theory. We also prove that the $\infty$-category of $\infty$-cosheaves is presentable and equivalent to an $\infty$-category of left adjoint functors.
\end{abstract}

\section*{Introduction}

Let $\X$ be a small category equipped with a pretopology $J$ and $\V$ be a category. 
Recall that a functor $\G : \X \to \V$ is called a \textit{precosheaf} on $\X$. 
It is called a \textit{cosheaf} if the diagram
\begin{equation}\label{coeq cosh}
\coprod_{j,k \in I} \G(U_j \times_U U_j) \rightrightarrows \coprod_{i \in I} \G(U_i) \to \G(U)
\end{equation}
is a coequalizer in $\V$ for all $\{U_i \to U\}_i \in J$. 
In other words, a cosheaf is a $\V^{\mathrm{op}}$-valued sheaf. 
This is a dual notion of sheaves. 
For example, the zeroth homology functor on the category of open sets is a cosheaf of abelian groups. 
We write $\pcsh(\X,\V)$ for the category of $\V$-valued precosheaves on $\X$ and $\csh(\X,\V)$ for its full subcategory consisting of cosheaves. 
In \cite{Pra}, Prasolov proved that the inclusion $\csh(\X,\V) \hookrightarrow \pcsh(\X,\V)$ admits a \textit{right} adjoint, called the cosheafification functor. This is a dual notion of sheafification, which is a \textit{left} adjoint functor of the natural inclusion. 

In this paper, we develop a basic theory of cosheaves in $\i$-category, 
and prove the existence of cosheafifications in the $\i$-categorical situation. 
Let $\X$ be a small $\i$-site and $\V$ be a symmetric closed monoidal $\i$-category which is presentable. 
We define $\V$-valued cosheaves on $\X$ by using \v{C}ech nerves instead of coequalizer diagrams (Definition \ref{def cosh}). 
The reason that we use \v{C}ech nerves is to define cosheaves as a dual of sheaves in the sense of Lurie \cite{HTT}. 
It turns out that this definition is indeed a generalization of cosheaves in $1$-category (see Remark \ref{ord cosh}). 

Our main result is summarized as follows.

\begin{thmi}[see Theorem \ref{main}]\label{main intro}
Assume that $\X$ and $\V$ satisfy the conditions \eqref{SL=L} and \eqref{sh=SVloc}. Then

(1) $\csh(\X,\V) \hookrightarrow \pcsh(\X,\V)$ admits a right adjoint,

(2) $\csh(\X,\V)$ is presentable, and

(3) $\csh(\X,\V)$ is equivalent to the $\i$-category of left adjoint functors from $\sh(\X,\V)$ to $\V$.
\end{thmi}

The right adjoint functor constructed in (1) gives the $\infty$-cosheafification. 
The conditions \eqref{SL=L} and \eqref{sh=SVloc} are satisfied in various cases including the following 
(cf.~Examples \ref{ex sl=l} and \ref{ex sh=svloc}): 
with $\X$ arbitrary,
\begin{itemize}
\item $\V$ the $\i$-category $\igrpd$ of $\i$-groupoids,
\item $\V$ the full subcategory of $\igrpd$ consisting of $n$-truncated objects for $n \geq 0$, 
\item $\V$ the connective part of the $\i$-derived category of a commutative ring,  
\item $\V$ the $\i$-category of connective spectra.
\end{itemize}
In \cite{BF}, Bunge-Funk proved that for a site $\X$ and an elementary topos $\V$, $\csh(\X,\V)$ is equivalent to the category of cocontinuous functors from $\sh(\X,\V)$ to $\V$. Theorem \ref{main intro} (3) implies Bunge-Funk's result in the case $\V = \set$.

This paper is organized as follows. In \S\ref{ndp}, we give technical preliminaries from the $\i$-category theory. In \S\ref{PPP}, we prove that the $\i$-category $\pcsh(\X,\V)$ is equivalent to the $\i$-category of left adjoint functors from $\psh(\X,\V)$ to $\V$ under the condition \eqref{SL=L}. In \S\ref{Pf main}, we prove the main result.

\begin{nota}
In this paper, we use quasi-categories as a model of $\i$-categories. We write $\igrpd$ for the $\i$-category of small $\i$-groupoids. For an $\i$-category $\C$, the simplicially enriched hom-set from $X \in \C$ to $Y \in \C$ is denoted by $\sHom_{\C}(X,Y)$. When $\C = \igrpd$, we write $[X,Y] = \sHom_{\igrpd}(X,Y)$ and $\P(X) = [X,\igrpd]$.
\end{nota}

\begin{acknow}
I would like to thank my adviser Shohei Ma for many useful advices. This work was supported by JSPS KAKENHI 19J21433.
\end{acknow}

\section{Right non-degenerate $\i$-bifunctors}\label{ndp}

In this section, we give technical preliminaries from the theory of $\i$-categories. For $\i$-categories $\C_1$ and $\C_2$, we write $[\C_1,\C_2]^\mathrm{L}$ (resp.~$[\C_1,\C_2]^\mathrm{R}$) for the full subcategory of $[\C_1,\C_2]$ consisting of left (resp.~right) adjoint functors. We refer to \cite{HTT} for the basic theory of $\i$-categories.

\begin{defi}
Let $\C$, $\C^\vee$ and $\V$ be $\i$-categories. A bifunctor $\C \times \C^\vee \to \V$ is called \textit{right non-degenerate}, if the induced functor $\C^\vee \to [\C,\V]$ is fully faithful and its essential image coincides with $[\C,\V]^\mathrm{L}$.
\end{defi}

Throughout this section, we fix a right non-degenerate bifunctor $\p{}{} : \C \times \C^\vee \to \V$.

\begin{lem}\label{pres. dual}
If $\C$ and $\V$ are presentable, then so is $\C^\vee$.
\end{lem}

\begin{proof}
This follows from the equivalence of $\i$-categories $\C^\vee \cong [\C,\V]^\mathrm{L}$ and \cite[Prop.~5.5.3.8]{HTT}.
\end{proof}

We write $\G^\d$ for the functor $\p{-}{\G} : \C \to \V$. Recall that there exists a canonical equivalence of $\i$-categories
\[
\e : [\C,\V]^\mathrm{L} \xrightarrow{\cong} \left([\V,\C]^\mathrm{R}\right)^{\mathrm{op}}
\]
such that $\e F$ is right adjoint to $F$ for all $F \in [\C,\V]^\mathrm{L}$ (see \cite[Prop.~5.2.6.2]{HTT} and its proof). Then we write $\G_\d = \e(\G^\d)$ for $\G \in \C^\vee$.

\begin{defi}
Let $\D$ be a full subcategory of $\C$. We write $\D^\vee$ for the full subcategory of $\C^\vee$ consisting of $\G \in \C^\vee$ such that $\G_\d v \in \D$ for all $v \in \V$. We call $\D^\vee$ \textit{the dual of} $\D$ with respect to $\p{}{}$.
\end{defi}

The following property says that localizations of $\C$ induce a right-nondegenerate bifunctor.

\begin{prop}\label{loc rndp}
Let $\D$ be a reflective full subcategory of $\C$.

(1) The restriction $\D \times \D^\vee \to \V$ of $\p{}{}$ is right non-degenerate.

(2) If $\C$, $\D$ and $\V$ are presentable, then $\D^\vee$ is coreflective in $\C^\vee$.
\end{prop}

\begin{proof}
(1) For $\G \in \D^\vee$, the functor $\G_\d : \V \to \C$ induces $\G_\ddagger : \V \to \D$. Since $\G_\ddagger$ is right adjoint to $\G^\d|_\D$, we obtain $(-)_{\ddagger} : \D^\vee \to ([\V,\C]^\mathrm{R})^{\mathrm{op}}$. By the $\i$-commutative diagram
\[
\begin{CD}
\D @>{(-)_{\ddagger}}>> \left([\V,\D]^\mathrm{R}\right)^{\mathrm{op}} \\
@VVV @VVV \\
\C @>>{\cong}> \left([\V,\C]^\mathrm{R}\right)^{\mathrm{op}},
\end{CD}
\]
the functor $(-)_{\ddagger}$ is fully faithful. The composition 
\[
{G' : \V \xrightarrow{G} \D \hookrightarrow \C}
\]
is contained in $[\V,\C]^\mathrm{R}$ for all $G \in [\V,\D]^\mathrm{R}$. Thus there exists $\G \in \C^\vee$ such that $\G_\d$ is equivalent to $G'$ in $([\V,\C]^\mathrm{R})^{\mathrm{op}}$. Then we have $\G_\ddagger \cong G$. This means that the functor $(-)_{\ddagger}$ is essentially surjective and thus equivalence of $\i$-categories. Therefore, $\D \times \D^\vee \to \V$ induces an equivalence of $\i$-categories $\D^\vee \xrightarrow{\cong} ([\V,\D]^\mathrm{R})^{\mathrm{op}} \xrightarrow{\cong} [\D,\V]^\mathrm{L}$.

(2) By Lemma \ref{pres. dual}, $\C^\vee$ and $\D^\vee$ are presentable. Thus by \cite[Cor.~5.5.2.9]{HTT}, it suffices to show that $\D^\vee$ is closed under small colimits in $\C^\vee$. Since limits of right adjoint functors are computed objectwise and $\D$ is closed under small limits in $\C$, the functor $[\V,\D]^\mathrm{R} \to [\V,\C]^\mathrm{R}$ preserves all small limits. Therefore, the assertion follows from the equivalences $(\C^\vee)^{\mathrm{op}} \cong [\V,\C]^\mathrm{R}$ and $(\D^\vee)^{\mathrm{op}} \cong [\V,\D]^\mathrm{R}$.
\end{proof}

Next, we consider the dual of the localization of $\C$ by a set of morphisms $S$. Let $\C^{S\mathchar`-\mathrm{loc}}$ be the full subcategory of $\C$ consisting of $S$-local objects (see \cite[Def.~5.5.4.1]{HTT}). We write $\C_{S\mathchar`-\mathrm{col}^\vee}$ for the full subcategory of $\C^\vee$ consisting of objects $\G$ such that for every $f : \F \to \F'$ in $S$, the induced morphism
\[
\p{f}{\G} : \p{\F}{\G} \to \p{\F'}{\G}
\]
is an equivalence in $\V$.

\begin{prop}\label{prop dually coloc}
Let $S$ be a set of morphisms in $\C$. Then $(\C^{S\mathchar`-\mathrm{loc}})^\vee = \C_{S\mathchar`-\mathrm{col}^\vee}$.
\end{prop}

\begin{proof}
Let $f : \F \to \F'$ be a morphism in $\C$ and let $\G \in \C^\vee$. It suffices to show that the following two conditions are equivalent.
\begin{enumerate}
\item\label{1 of d col} The morphism $\G^\d f : \G^\d\F \to \G^\d\F'$ is an equivalence in $\V$.

\item\label{2 of d col} The morphism $\sHom_\C(f,\G_\d v) : \sHom_\C(\F',\G_\d v) \to \sHom_\C(\F,\G_\d v)$ is an equivalence in $\igrpd$ for all $v \in \V$.
\end{enumerate}
By the $\i$-Yoneda lemma, the condition \eqref{1 of d col} is equivalent to the following.
\begin{itemize}
\item[(\ref{1 of d col}')] The morphism $\sHom_\C(\G^\d f,v) : \sHom_\C(\G^\d\F,v) \to \sHom_\C(\G^\d\F',v)$ is an equivalence in $\igrpd$ for all $v \in \V$.
\end{itemize}
Since $\G^\d$ is left adjoint to $\G_\d$, the equivalence $\sHom_\C(\G^\d f,v) \cong \sHom_\C(f,\G_\d v)$ shows that (\ref{1 of d col}') $\Leftrightarrow$ (\ref{2 of d col}).
\end{proof}

\begin{rem}\label{rmk dloc}
When $\C$ is presentable and $S$ is small, $\C^{S\mathchar`-\mathrm{loc}}$ is reflective and can be regarded as the localization $S^{-1}\C$ of $\C$ by $S$ (see \cite[Prop.~5.5.4.2 and 5.5.4.15]{HTT}). Therefore, we obtain a right non-degenerate bifunctor $S^{-1}\C \times \C_{S\mathchar`-\mathrm{col}^\vee} \to \V$ by Proposition \ref{loc rndp}.
\end{rem}

\section{Pairing between presheaves and precosheaves}\label{PPP}

Throughout this section, we fix a small simplicial set $\X$ and (the underlying $\i$-category of) a symmetric closed monoidal $\i$-category $\V$ which is complete and cocomplete. We write $\psh(\X,\V) = [\X^{\mathrm{op}},\V]$ and $\pcsh(\X,\V) = [\X,\V]$. Our aim of this section is to give a canonical right non-degenerate bifunctor $\psh(\X,\V) \times \pcsh(\X,\V)\to \V$. Let $\otimes$, $\iHom$ and $\1$ denote the monidal product, the internal hom and the monoidal unit of $\V$, respectively. Let $r_*$ be the covariant functor $\V \to \igrpd$ represented by $\1$ and $r^*$ be the functor $\igrpd \to \V$ defined by $S_\bu \mapsto \mathop{\mathrm{colim}}(S_\bu \to \{\1\} \hookrightarrow \V)$. The pair $(r^*,r_*)$ is a variation of free-forgetful adjunctions.

\begin{lem}\label{adj r}
The functor $r_*$ is right adjoint to $r^*$.
\end{lem}

\begin{proof}
For each $S_\bu \in \igrpd$, $v \in \V$ and $n \geq 0$, there exist equivalences
\[
\sHom_\V(r^*S_n,v) \cong \sHom_\V\left(\prod_{s \in S_n}\1,v\right) \cong \coprod_{s \in S_n} \sHom_\V(\1,v) \cong \sHom_{\igrpd}(S_n,r_*v).
\]
Since $S_\bu$ is equivalent to the colimit of the functor $\Delta \to \igrpd$, $[n] \mapsto S_n$, and $r^*$ preserves colimits, we obtain an equivalence
\[
\sHom_\V(r^*S_\bu,v) \cong \sHom_{\igrpd}(S_\bu,r_*v)
\]
which is functorial in both $S_\bu$ and $v$.
\end{proof}

Let $\yo^{\i}$ be the $\i$-Yoneda embedding $\X \hookrightarrow \P(\X)$. By Lemma \ref{adj r}, $r^*$ and $r_*$ induce an adjunction
\[
\overline{r} : \P(\X) \rightleftarrows \psh(\X,\V) : \underline{r}.
\]
We define $\yo^\V$ as the composition $\overline{r} \circ \yo^\i$. We give examples of $r^*$ and $r_*$ for some $\V$.

\begin{examp}\label{ex r_* r^*}
(1) Assume that $\V = \igrpd$ equipped with the cartesian monoidal structure. Then $r^*$ and $r_*$ are equivalent to the identity. Thus we obtain that $\yo^\V \cong \yo^\i$.

(2) Assume that $\V$ is the full subcategory category $\igrpd_{\leq n}$ of $\igrpd$ consisting of $n$-truncated objects for $n \geq 0$. Then $r_*$ is the inclusion $\igrpd_{\leq n} \hookrightarrow \igrpd$ and $r^*$ is the truncation functor.

(3) Let $\R$ be a commutative unital ring. Assume that $\V$ is the derived $\i$-category $\D(\R)$ of $\R$ (see \cite{HA}). Recall that $\D(\R)$ has a monoidal structure induced by derived tensor products. Under this monoidal structure, we obtain that $r_*(C_{\bu}) \cong \mathrm{DK}(\tau_{\geq 0}C_{\bu})$ for each $C_{\bu} \in \D(\R)$, where $\mathrm{DK}$ means the Dold-Kan correspondence and $\tau_{\geq 0}$ means the truncation. Thus $r^*$ coincides with the functor of singular chain complexes.

(4) Assume that $\V$ is the stable $\i$-category of spectra $\mathbf{Sp}$. Then $\V$ has a symmetric monoidal structure given by smash products (see \cite[\S 4.8.2]{HA}). Since $\1$ is the sphere spectrum, $r^*$ is the suspension ${\Sigma^\i : \igrpd \to \mathbf{Sp}}$ and $r_*$ is the looping $\Omega^\i : \mathbf{Sp} \to \igrpd$.
\end{examp}

\begin{lem}\label{riHom = sHom}
There exists an equivalence $\sHom_\V \cong r_* \circ \iHom$ in $[\V^{\mathrm{op}} \times \V,\igrpd]$.
\end{lem}

\begin{proof}
For each $v,w \in \V$, there exist equivalences
\[
r_*\iHom(v,w) = \sHom_{\V}(\1,\iHom(v,w)) \cong \sHom_{\V}(v \otimes \1,w) \cong \sHom_{\V}(v,w)
\]
which are functorial in both $v$ and $w$.
\end{proof}

For a bifunctor of $\i$-categories $F : \C^{\mathrm{op}} \times \C \to \D$, we write $\int_{c \in \C}F(c,c)$ (resp.~$\int^{c \in \C}F(c,c)$) for the $\i$-categorical end (resp.~coend) of $F$ introduced by Glasman \cite[Def.~2.2]{Gla}. Since ends (resp.~coends) are defined as a limit (resp.~colimit), these commute with right (resp.~left) adjoint functors. We define a bifunctor $\iHom' : \psh(\X,\V)^{\mathrm{op}} \times \psh(\X,\V) \to \V$ by
\[
\iHom'(\F,\F') = \int_{U \in \X} \iHom(\F(U),\F'(U))
\]
for $\F,\F' \in \psh(\X,\V)$. The next lemma is a variation of Yoneda's lemma for $\iHom'$.

\begin{lem}\label{yoneda ihom}
In $[\X^{\mathrm{op}} \times \psh(\X,\V), \V]$, the functor $(U,\F) \mapsto \F(U)$ is equivalent to
\[
(U,\F) \mapsto \iHom'(\yo^\V(U),\F).
\]
\end{lem}

\begin{proof}
Since ends are defined as a limit, these commute with hom-functors. Thus we obtain an equivalence
\[
\sHom_{\V}(v,\iHom'(\yo^\V(U),\F)) \cong \int_{V \in \X} \sHom_\V(v,\iHom(\yo^\V(U)(V),\F(V)))
\]
for each $U \in \X$, $\F \in \psh(\X,\V)$ and $v \in \V$. We have
\begin{align*}
\sHom_\V(v,\iHom(\yo^\V(U)(V),\F(V))) &\cong \sHom_\V(v \otimes \yo^\V(U)(V), \F(V)) \\
&\cong \sHom_\V(\yo^\V(U)(V),\iHom(v,\F(V))).
\end{align*}
By Lemmas \ref{adj r} and \ref{riHom = sHom}, we also have
\begin{align*}
\sHom_\V(\yo^\V(U)(V),\iHom(v,\F(V))) &\cong \sHom_{\igrpd}(\yo^\i(U)(V),r_*\iHom(v,\F(V))) \\
&\cong \sHom_{\igrpd}(\yo^\i(U)(V),\sHom_\V(v,\F(V))).
\end{align*}
Hence, we obtain
\[
\sHom_{\V}(v,\iHom(\yo^\V(U),\F)) \cong \int_{V \in \X} \sHom_{\igrpd}(\yo^\i(U)(V),\sHom_\V(v,\F(V))).
\]
The right-hand side is equivalent to $\sHom_{\P(\X)}(\yo^\i(U),\sHom_\V(v,\F(-)))$ by \cite[Prop.~2.3]{Gla}. Therefore, the $\i$-Yoneda lemma in $\X$ gives an equivalence
\[
\int_{V \in \X} \sHom_\V(\yo^\i(U)(V),\sHom_\V(v,\F(V))) \cong \sHom_\V(v,\F(U)).
\]
Applying the $\i$-Yoneda lemma in $\V$ to the equivalence
\[
\sHom_{\V}(v,\iHom(\yo^\V(U),\F)) \cong \sHom_\V(v,\F(U)),
\]
we obtain the assertion.
\end{proof}

\begin{defi}\label{ppp}
We define a bifunctor $\p{\cdot}{\cdot} : \psh(\X,\V) \times \pcsh(\X,\V) \to \V$ by
\begin{equation}\label{pppeq}
\p{\F}{\G} = \int^{U \in \X} \F(U) \otimes \G(U)
\end{equation}
for $\F \in \psh(\X,\V)$ and $\G \in \pcsh(\X,\V)$. We write $\G^\d$ for the functor $\p{-}{\G} : \psh(\X,\V) \to \V$, and $\G_\d$ for the functor $\V \to \psh(\X,\V)$ defined by $v \mapsto \iHom(\G(-),v)$.
\end{defi}

We are interested in the non-degeneracy of the paring $\p{\cdot}{\cdot}$, for which our main result is Proposition \ref{psh-pcsh nd}. We first prepare some lemmas.

\begin{lem}\label{adj G_d G^d}
In $[\psh(\X,\V)^{\mathrm{op}} \times \pcsh(\X,\V)^{\mathrm{op}} \times \V, \V]$, the functor
\[
(\F,\G,v) \mapsto \iHom(\p{\F}{\G},v)
\]
is equivalent to
\[
(\F,\G,v) \mapsto \iHom'(\F,\G_\d v).
\]
\end{lem}

\begin{proof}
For each $\F \in \psh(\X,\V)$, $\G \in \pcsh(\X,\V)$ and $v \in \V$, there exist equivalences
\begin{align*}
\iHom(\p{\F}{\G},v) &\cong \int_{U \in \X} \iHom(\F(U) \otimes \G(U),v) \\
&\cong \int_{U \in \X} \iHom(\F(U),\iHom(\G(U),v)) \\
&= \int_{U \in \X} \iHom(\F(U),(\G_\d v)(U)) \\
&= \iHom'(\F,\G_\d v)
\end{align*}
which are functorial in $\F$, $\G$ and $v$.
\end{proof}

By Lemmas \ref{riHom = sHom} and \ref{adj G_d G^d}, we see that $\G^\d$ is left adjoint to $\G_\d$. This means that the pairing \eqref{pppeq} induces a functor $(-)^\d : \pcsh(\X,\V) \to [\psh(\X,\V),\V]^\mathrm{L}$. Let $\mathrm{Res}$ be the functor $[\psh(\X,\V),\V] \to \pcsh(\X,\V)$ defined by $F \mapsto F \circ \yo^\V$. We prove that $(-)^\d$ is an essential section of $\mathrm{Res}$.

\begin{lem}\label{res section}
The composition
\[
\pcsh(\X,\V) \xrightarrow{(-)^\d} [\psh(\X,\V),\V] \xrightarrow{\mathrm{Res}} \pcsh(\X,\V)
\]
is equivalent to the identity of $\pcsh(\X,\V)$. In other words, we have
\[
\p{\yo^\V(U)}{\G} \cong \G(U)
\]
which is functorial in $U \in \X$ and $\G \in \pcsh(\X,\V)$.
\end{lem}

\begin{proof}
For each $\G \in \pcsh(\X,\V)$ and $v \in \V$, there exist equivalences
\begin{align*}
\iHom(\G^\d \circ \yo^\V(-),v) \cong \iHom'(\yo^\V(-),\G_\d v) \cong (\G_\d v)(-) = \iHom(\G(-),v)
\end{align*}
by Lemmas \ref{adj G_d G^d} and \ref{yoneda ihom}. Thus Lemma \ref{riHom = sHom} gives
\[
\sHom_{\V}(\G^\d \circ \yo^\V(-),v) \cong \sHom_{\V}(\G(-),v).
\]
Therefore, the $\i$-Yoneda lemma in $\V$ leads to a functorial equivalence $\mathrm{Res}(\G^\d) \cong \G$.
\end{proof}

We prove a variation of the co-Yoneda lemma.

\begin{lem}\label{coyoneda ihom}
In $[\X^{\mathrm{op}} \times \psh(\X,\V), \V]$, the functor $(V,\F) \mapsto \F(V)$ is equivalent to
\[
V \mapsto \int^{U \in \X} \F(U) \otimes \yo^\V(U)(V)
\]
\end{lem}

\begin{proof}
Since $\psh(\X,\V) = \pcsh(\X^{\mathrm{op}},\V)$ and $\pcsh(\X,\V) = \psh(\X^{\mathrm{op}},\V)$, Lemma \ref{res section} for $\X^{\mathrm{op}}$ gives a desired equivalence.
\end{proof}

 We write $[\psh(\X,\V),\V]^\spadesuit$ for the full subcategory of $[\psh(\X,\V),\V]$ consisting of functors $F$ such that the composition 
\[
\psh(\X,\V) \times \V \xrightarrow{(F,\mathrm{id})} \V \times \V \xrightarrow{\iHom} \V
\]
is equivalent to
\[
\psh(\X,\V) \times \V \xrightarrow{(\mathrm{id},G)} \psh(\X,\V) \times \psh(\X,\V) \xrightarrow{\iHom'} \V
\]
in $[\psh(\X,\V)\times\psh(\X,\V),\V]$. Lemma \ref{riHom = sHom} shows that $[\psh(\X,\V),\V]^\spadesuit \subseteq [\psh(\X,\V),\V]^\mathrm{L}$.

\begin{lem}\label{st adj otimes}
A functor $F : \psh(\X,\V) \to \V$ belongs to $[\psh(\X,\V),\V]^\spadesuit$ if and only if the following diagram is $\i$-commutative:
\[
\begin{CD}
\V \times \psh(\X,\V) @>\otimes>> \psh(\X,\V) \\
@V(\mathrm{id},F)VV @VVFV \\
\V \times \psh(\X,\V)@>>\otimes> \psh(\X,\V).
\end{CD}
\]
\end{lem}

\begin{proof}
For $v,w \in \V$, $\F \in \psh(\X,\V)$ and an adjoint pair ${F : \psh(\X,\V) \rightleftarrows \V : G}$, we have $\sHom_{\V}(v \otimes F(\F),w) \cong \sHom_{\V}(v,\iHom(F(\F),w))$ and
\[
\sHom_{\V}(F(v \otimes \F),w) \cong \sHom_{\V}(v \otimes \F,G(w)) \cong \sHom_{\V}(v,\iHom'(\F,G(w))).
\]
Thus the assertion follows from the $\i$-Yoneda lemma in $\V$.
\end{proof}

Now we can prove our non-degeneracy result on the pairing $\p{\cdot}{\cdot}$.

\begin{prop}\label{psh-pcsh nd}
The functor $(-)^\d : \pcsh(\X,\V) \to [\psh(\X,\V),\V]^\spadesuit$ is an equivalence of $\i$-categories and $\mathrm{Res}$ gives its inverse. In particular, when \begin{equation}\label{SL=L}
[\psh(\X,\V),\V]^\spadesuit = [\psh(\X,\V),\V]^\mathrm{L}
\end{equation}
holds, the pairing $\p{\cdot}{\cdot}$ is right non-degenerate.
\end{prop}

\begin{proof}
By Lemma \ref{res section}, it suffices to show that the composition
\[
[\psh(\X,\V),\V]^\mathrm{L} \xrightarrow{\mathrm{Res}} \pcsh(\X,\V) \xrightarrow{(-)^\d} [\psh(\X,\V),\V]^\mathrm{L}
\]
is equivalent to the identity of $[\psh(\X,\V),\V]^\spadesuit$. For $F \in [\psh(\X,\V),\V]^\mathrm{L}$ and $\F \in \psh(\X,\V)$, we have
\begin{align*}
F(\F) &\cong F\left( \int^{U \in \X} \F(U) \otimes \yo^\V(U)(-) \right) \\
&\cong \int^{U \in \X} F\left( \F(U) \otimes \yo^\V(U)(-) \right)
\end{align*}
by Lemma \ref{coyoneda ihom}. We also have
\[
\mathrm{Res}(F)^\d(\F) = \int^{U \in \X} \F(U) \otimes F\left(\yo^\V(U)\right)
\]
by the definition of $\mathrm{Res}(F)$. Thus the assertion follows from Lemma \ref{st adj otimes}.
\end{proof}

Finally, we give some examples of $\V$ satisfying the condition \eqref{SL=L}.

\begin{examp}\label{ex sl=l}
(1) Clearly, $\igrpd$ and $\igrpd_{\leq n}$ for $n \geq 0$ satisfy \eqref{SL=L}.

(2) For a commutative unital ring $\R$, let $\D_{\geq 0}(\R)$ be the full subcategory of $\D(\R)$ consisting of connective complexes. Then the Dold-Kan correspondence says that $\D_{\geq 0}(\R)$ satisfies \eqref{SL=L}. Moreover, for all $C_{\bullet} \in \D(\R)$ and ${\F \in \psh(\X,\D(\R))}$, we have natural equivalences
\begin{align*}
F(C_{\bullet} \times \F) &\cong F((\colim{n \in \mathbb{Z}}(\tau_{\geq n} C_{\bullet})[n]) \otimes (\colim{m \in \mathbb{Z}} (\tau_{\geq m}\F)[m])) \\
&\cong \colim{n,m \in \mathbb{Z}} F((\tau_{\geq n} C_{\bullet})[n] \otimes (\tau_{\geq m}\F)[m])\\
&\cong \colim{n,m \in \mathbb{Z}} ((\tau_{\geq n} C_{\bullet})[n] \otimes F(\F[m])) \\
&\cong (\colim{n \in \mathbb{Z}}(\tau_{\geq n} C_{\bullet})[n]) \otimes F(\colim{m \in \mathbb{Z}} (\tau_{\geq m}\F[m]))\\
&\cong C_{\bullet} \otimes F(\F)
\end{align*}
by Lemma \ref{st adj otimes} in $\D_{\geq 0}(\R)$. Thus $\D(\R)$ also satisfies \eqref{SL=L}.

(3) The $\i$-category $\mathbf{Sp}$ satisfies \eqref{SL=L}. Indeed, for $E_{\bullet} \in \mathbf{Sp}$, $\F \in \psh(\X,\mathbf{Sp})$ and $F \in [\psh(\X,\mathbf{Sp}),\mathbf{Sp}]^\mathrm{L}$, we have
\begin{align*}
F(E_{\bullet} \wedge \F) &\cong F(\colim{n\in \mathbb{Z}} (\Sigma^{\i} E_n)[-n] \wedge \F) \\
&\cong \colim{n\in \mathbb{Z}} F(\Sigma^{\i} E_n \wedge \F[n])([-n]) \\
&\cong \colim{n\in \mathbb{Z}} F(\colim{E_n} \1 \wedge \F[n])([-n]) \\
&\cong \colim{n\in \mathbb{Z}} (\colim{E_n} F(\1 \wedge \F[n]))([-n]) \\
&\cong \colim{n\in \mathbb{Z}} (\colim{E_n} F(\F[n]))([-n]) \\
&\cong \colim{n\in \mathbb{Z}} (\colim{E_n} \1 \wedge F(\F[n]))[-n] \\
&\cong  ((\colim{n\in \mathbb{Z}} \Sigma^{\i} E_n)[-n]) \wedge F(\F) \\
&\cong E_{\bullet} \wedge F(\F),
\end{align*}
where $\colim{E_n}$ means the colimit of the constant diagram on $E_n$. Similarly, the full subcategory $\mathbf{Sp}_{\geq 0}$ of $\mathbf{Sp}$ consisting of connective spectra also satisfies \eqref{SL=L}.
\end{examp}

\section{Proof of the main result}\label{Pf main}

Let $\X$ be a small $\i$-category equipped with a Grothendieck topology $J$ in the sense of \cite[Def.~6.2.2.1]{HTT} and $\V$ be a presentable symmetric closed monoidal $\i$-category. In this section we prove the main result stated in the introduction. We first give a definition of $\V$-valued (co)sheaves on $\X$. Recall that a family of morphisms $\{U_i \to X\}$ is called a \textit{covering family} of $X \in \X$ if the smallest sieve of $\X_{/X}$ containing $\{U_i \to X\}$ is a covering sieve. For each covering family $\U_0$, we write $\check{C}(\mathcal{\U}_0) : \Delta^{\mathrm{op}} \to \P(\X)$ for its \v{C}ech nerve and $\check{C}^\V(\mathcal{\U}_0) = \overline{r} \circ \check{C}(\mathcal{\U}_0)$.

We define sheaves and cosheaves by using \v{C}ech nerves, instead of more familiar definition via (co)equalizer diagrams. 
It seems that this is more suitable for the theory of $\i$-categories as developed in \cite{HTT}. 
The two definitions agree (Remark \ref{ord cosh}).

\begin{defi}\label{def cosh}
(1) A presheaf $\F \in \psh(\X,\V)$ is called a \textit{sheaf} if for each covering family $\U_0$ of $\X$, the composition
\[
\Delta \xrightarrow{\check{C}^\V(\mathcal{\U}_0)^{\mathrm{op}}} \psh(\X,\V)^{\mathrm{op}} \xrightarrow{\iHom'(-,\F)} \V 
\]
is a limit diagram in $\V$. We write $\sh(\X,\V)$ for the full subcategory of $\psh(\X,\V)$ consisting of sheaves.

(2) A precosheaf $\G \in \pcsh(\X,\V)$ is called a \textit{cosheaf} if for each covering family $\U_0$ of $\X$, the composition
\[
\Delta^{\mathrm{op}} \xrightarrow{\check{C}^\V(\mathcal{\U}_0)} \psh(\X,\V) \xrightarrow{\G^\d} \V 
\]
is a colimit diagram in $\V$. We write $\csh(\X,\V)$ for the full subcategory of $\pcsh(\X,\V)$ consisting of cosheaves.
\end{defi}

\begin{rem}
Assume that $\X$ has all pullbacks. Then a \v{C}ech nerve of a covering family of $\X$ can be regarded as a simplicial object in $\X$. Thus by Lemma \ref{yoneda ihom}, $\F \in \psh(\X,\V)$ is a sheaf if and only if $\F \circ \check{C}(\mathcal{\U}_0)^{\mathrm{op}}$ is a limit diagram for every covering family $\U_0$ of $\X$. Similarly, $\G \in \pcsh(\X,\V)$ is a cosheaf if and only if $\G \circ \check{C}(\mathcal{\U}_0)$ is a colimit diagram for every $\U_0$ by Lemma \ref{res section}.
\end{rem}

For a monomorphism $i_\U : j(\U) \to \yo^\i$ associated with a covering sieve $\U$ (see \cite[Prop.~6.2.2.5]{HTT}), we write $i^\V_\U = \overline{r}(i_\U)$ and $\yo^\V(\U) = \overline{r}(j(\U))$. We write $S_{\V} = \{i^\V_\U \in \mathcal{M}or(\psh(\X,\V))\ |\ \U \in J\}$. The following lemma assures that our definition of sheaves coincides with that of \cite[Def.~6.2.2.6]{HTT} when $\V=\igrpd$.

\begin{lem}\label{SV-local}
A presheaf $\F \in \psh(\X,\V)$ is a sheaf if and only if the morphism $\iHom'(i^\V_\U,\F)$ is an equivalence in $\V$ for all $i^\V_\U \in S_\V$. Moreover, when $r_*$ is conservative, i.e., it reflects equivalences, then $\F$ is a sheaf if and only if it is $S_\V$-local.
\end{lem}

\begin{proof}
Let $\Delta^{\geq 1}$ for the full subcategory of $\Delta$ consisting of $[n]$ for $n \geq 1$. Since $\overline{r}$ preserves colimits, we have
\[
\mathop{\mathrm{colim}} \check{C}^\V(\mathcal{\U}_0)|_{\Delta^{\geq 1}} \cong \overline{r}(\mathop{\mathrm{colim}} \check{C}(\mathcal{\U}_0)|_{\Delta^{\geq 1}}) \cong \overline{r}(j(\U)) = \yo^\V(\U).
\]
Thus the first half of the assertion follows from Lemma \ref{yoneda ihom} and the second half follows from Lemma \ref{riHom = sHom}.
\end{proof}

Now we prove the main result. 
This can be applied to the $\i$-categories in the following Example \ref{ex sh=svloc}. 

\begin{thm}\label{main}
Assume that $\X$ and $\V$ satisfy \eqref{SL=L} and also the following condition:
\begin{equation}\label{sh=SVloc}
\sh(\X,\V) = \sh(\X,\V)^{S_{\V}\mathchar`-\mathrm{loc}}.
\end{equation}
Then the following hold.

(1) Under the pairing \eqref{pppeq}, we have $\csh(\X,\V) = \sh(\X,\V)^\vee$.

(2) The $\i$-category $\sh(\X,\V)$ is presentable.

(3) The $\i$-category $\csh(\X,\V)$ is presentable.

(4) The inclusion $\sh(\X,\V) \hookrightarrow \psh(\X,\V)$ admits a left adjoint.

(5) The inclusion $\csh(\X,\V) \hookrightarrow \pcsh(\X,\V)$ admits a right adjoint.

(6) There exists an equivalence of $\i$-categories $\csh(\X,\V) \cong [\sh(\X,\V),\V]^\mathrm{L}$.
\end{thm}

\begin{proof}
The assertion (2) follows from Lemma \ref{SV-local} and \cite[Prop.~5.5.4.15]{HTT}. Moreover, (1) and (2) imply (3)-(6). Indeed, since $\sh(\X,\V)$ is closed under limits in $\psh(\X,\V)$, (4) follows from (2) and \cite[Cor.~5.5.2.9]{HTT}. Therefore, (3), (5) and (6) follow from Lemma \ref{pres. dual} and Proposition \ref{loc rndp}. Thus it suffices to prove  (1). Since $\G^\d$ preserves colimits, a precosheaf $\G \in \pcsh(\X,\V)$ is a cosheaf if and only if the canonical morphism
\[
\p{\yo^\V(\U)}{\G} \to \p{\yo^\V(U)}{\G}
\]
is an equivalence in $\V$ by Lemma \ref{res section}. Thus we obtain (1) by Proposition \ref{prop dually coloc}.
\end{proof}

Most $\i$-categories from Example \ref{ex sl=l} satisfy the condition \eqref{sh=SVloc}.

\begin{examp}\label{ex sh=svloc}
(1) Clearly, $\igrpd$ and $\igrpd_{\leq n}$ for $n \geq 0$ satisfy \eqref{sh=SVloc}.

(2) The $\i$-category $\mathcal{D}_{\geq 0}(\R)$ in Example \ref{ex sl=l} (2) satisfies \eqref{sh=SVloc} by Lemma \ref{SV-local}. Indeed, $r_*$ is conservative by the Dold-Kan correspondence.

(3) The $\i$-category $\mathbf{Sp}_{\geq 0}$ in Example \ref{ex sl=l} (3) satisfies \eqref{sh=SVloc} by Lemma \ref{SV-local}. Indeed, $r_*$ is conservative by \cite[Cor.~7.1.4.13]{HA}.
\end{examp}

\begin{rem}\label{ord cosh}
Recall that if $\X$ is a $1$-category and $\V = \set$, cosheaves are usually defined via coequalizer diagrams (cf.~\eqref{coeq cosh}). Theorem \ref{main} (1) says that our Definition \ref{def cosh} (2) is equivalent to this more familiar definition for such $(\X,\V)$. Indeed, we can prove that the category of cosheaves defined by coequalizer diagrams is the dual of the category of sheaves by a similar proof. Similarly, our definition is also a generalization of ordinary cosheaves of modules.
\end{rem}

\begin{rem}
In the context of $\i$-topos theory, sheaves on a simplicial set $\X$ are sometime defined as an object of an $\i$-topos $\mathcal{T}$ on $\X$. Note that there exists an $\i$-topos which is not the $\i$-category of sheaves on a Grothendieck topology on $\X$, unlike the ordinary topos theory. Then Theorem \ref{main} suggests a definition of cosheaves in such a situation. Namely, we can define a cosheaf on $\mathcal{T}$ as an object of the dual $\mathcal{T}^\vee$, or equivalently, a left adjoint functor $\mathcal{T} \to \igrpd$.
\end{rem}

\end{document}